\documentclass{amsart}
\usepackage[margin=1in]{geometry}
\usepackage{amsmath, amsfonts, amssymb, amsthm}
\usepackage{graphicx}
\usepackage{hyperref}
\usepackage{color}
\usepackage{xcolor}
\usepackage[framemethod=default]{mdframed}
\usepackage{dsfont}

\global\mdfdefinestyle{exampledefault}{%
linecolor=lightgray,linewidth=1pt,%
leftmargin=1cm,rightmargin=1cm,
}

\newtheorem{thm}{Theorem}[section]
\newtheorem{prop}[thm]{Proposition}
\newtheorem{cor}[thm]{Corollary}

\newtheorem{lemma}[thm]{Lemma}

\newtheorem{preremark}[thm]{Remark}
\newenvironment{remark}{\begin{preremark}\rm}{\medskip \end{preremark}}
\numberwithin{equation}{section}

\newcommand{\abs}[1]{\left\vert#1\right\vert}

\newcommand{\R}{\mathbb R}
\newcommand{\eps}{\varepsilon}

\newcommand{\grad} {\nabla}
\newcommand{\lap} {\Delta}

\newcommand{\dd} {\; \mathrm{d}}

\DeclareMathOperator{\supp}{supp}

\title{Upper bounds for parabolic equations and the Landau equation}
\author{Luis Silvestre}

\thanks{LS was partially supported by NSF grant DMS-1254332.}

\begin{document}
\begin{abstract}
We consider a parabolic equation in nondivergence form, defined in the full space $[0,\infty) \times \R^d$, with a power nonlinearity as the right hand side. We obtain an upper bound for the solution in terms of a weighted control in $L^p$. This upper bound is applied to the homogeneous Landau equation with moderately soft potentials. We obtain an estimate in $L^\infty(\R^d)$ for the solution of the Landau equation, for positive time, which depends only on the mass, energy and entropy of the initial data.
\end{abstract}

\maketitle

\section{Introduction}

In this article we analyze upper bounds for parabolic equations in nondivergence form, like
\[ f_t - a_{ij}(t,x) \partial_{ij} f \leq f^{1+\alpha}.\]

Our main purpose is to use these upper bounds to derive a priori estimates for the homogeneous Landau equation. In this particular equation, one can obtain upper and lower bounds for the ellipticity coefficients of $a_{ij}(t,x)$ that degenerate as $|x| \to \infty$. Moreover, we have an a priori estimate for $f$ in $L^\infty((0,+\infty),L_2^1(\R^d))$. With this objective in mind we study general parabolic equations  in the full space $[0,+\infty) \times \R^d$ allowing the ellipticity of the coefficients $a_{ij}(t,x)$ to degenerate at a specific rate as $x$ goes to infinity. We obtain an estimate for $\|f(t,\cdot)\|_{L^\infty}$ for any $t>0$ assuming that we have a priori an estimate on the weighted $L^p_\kappa$ space for appropriate values of $p$ and $\kappa$.

The main result is the following.
\begin{thm} \label{t:local-max-pple-intro}
Let $f : [0,1] \times \R^d \to \R$ be a function satisfying the following inequality in the classical sense
\[ f_t \leq a_{ij}(t,x) \partial_{ij} f +  C \|f(t,\cdot)\|_{L^\infty}^{1+\alpha}.\]
Here, $ \alpha$ is a parameter in $[0,2p/d))$, the coefficients $a_{ij}$ are locally uniformly elliptic and, for some constants $\delta>0$, $\Lambda>0$, $\kappa \in \R$ and $\beta \geq -\kappa/d$ and $p \in [1,\infty)$,
\begin{align*}
\det (a_{ij}) &\geq \delta (1+|x|)^{\beta d - \kappa},  \\
(a_{ij}(t,x)) &\leq \Lambda (1+|x|)^{\min(2\beta,2)} \mathrm{I},  \\
\int_{\R^d} (1+|x|)^\kappa f(t,x)^p \dd x &\leq N  \qquad \text{for all } t \geq 0.
\end{align*}

Then, for all $(t,x) \in [0,T] \times \R^d$, we have
\[ f(t,x) < \begin{cases}
K t^{-\frac d {2p}}, & \text{ if } t \leq T, \\
K T^{-\frac d {2p}}, & \text{ if } t > T.
\end{cases}
\]
Here $K$ and $T$ are constants depending on $\delta$, $\Lambda$, $N$, and the dimension $d$.
\end{thm}

A more precise description, including the explicit formula for $K$ and $T$, is given in Theorem \ref{t:local-max-pple-nonlinear}. We also provide a version of the main theorem for the case $p=1$ and $\alpha =2/ d$ in Theorem \ref{t:local-max-pple-nonlinear-log}. This bordeline situation requires a $\log$-correction in the right-hand side of the equation.

The restriction $\alpha \in [0,2p/d)$ is necessary for the result in Theorem \ref{t:local-max-pple-intro} to hold. We show in Proposition \ref{p:unimprovability} that the $L^\infty$ estimate would fail otherwise.

It is well known that the nonlinear heat equation
\[ f_t - \lap f = f^{1+\alpha},\]
may blow up in finite time for $\alpha > 0$. In the context of this paper, we study subsolutions of the above equation with the extra hypothesis that $\|f(t,\cdot)\|_{L^p}$ stays bounded at all times and derive an a priori estimate in $L^\infty$.  There are several ways to prove this estimate for the nonlinear heat equation. The interesting feature of the result in this article is that it applies to equations with rough coefficients $a_{ij}$ is nondivergence form, depending only of the ellipticity estimates of $\{a_{ij}\}$. Estimates for equations with rough coefficients are applicable to nonlinear equations, whose coefficients depend on the solution. Our main motivation is in the study of the homogeneous Landau equation.

\subsection{The Landau equation}
The Landau equation is a common model is plasma physics \cite{MR0258399}, \cite{lifshitz1981physical}. It is obtained as a limit of the Boltzmann equation when the angular singularity (often written as a parameter $\nu$) converges to two.

The results in this article apply to the space homogeneous case. The equation concerns a nonnegative function $f(t,v)$ (representing a density of particles) which satisfies the equation
\begin{equation} \label{e:intro-landau}
 f_t = \bar a_{ij}(t,v) \partial_{ij} f + \bar c(t,v) f .
\end{equation}

Here, $\bar a_{ij}(t,v)$ and $\bar c$ stand for
\begin{align}
\bar a_{ij}(t,v) &= a_{d,\gamma} \int_{\R^d} \left( I - \frac{w}{|w|} \otimes \frac{w}{|w|} \right) |w|^{\gamma+2} f(v-w) \dd w, \label{e:intro-aij} \\
\bar c(t,v) &= -\partial_{ij} \bar a_{ij} = c_{d,\gamma} \int_{\R^d} |w|^\gamma f(v-w). \label{e:intro-c}
\end{align}

The parameter $\gamma$ is an arbitrary number within the range $[-d,+\infty)$. In the endpoint case $\gamma=-d$, the last integral formula should be replaced by $\bar c = c_{d,\gamma} f$.

The Landau equation can also be written in divergence form as
\[ f_t = \partial_i \left( \bar a_{ij} f_{ij} - \bar  b_i f\right),\]
where $\bar b_i = \partial_j \bar a_{ij}$. In this paper, we use its nondivergence formulation.

Obtaining a priori estimates for the function $f$ seems to be harder for smaller values of $\gamma$. The case $\gamma \in [0,1]$ (called \emph{Maxwell molecules} when $\gamma = 0$ and \emph{hard potentials} when $\gamma > 0$) is essentially well understood (see \cite{desvillettes2000spatially} and \cite{villani1998spatially}). In that range, the solution becomes immediately $C^\infty$ for $t>0$ provided that the initial data has finite mass and energy. The theory available for the case $\gamma < 0$ (called \emph{soft potentials}) is less complete. For $\gamma \in [-2,0)$ (this case is called \emph{moderately soft potentials}), in \cite{alexandre2013some}, the authors prove that if the initial data $f_0 = f(0,\cdot)$ is in $L^2$, then $f(t,\cdot)$ stays in $L^2$ for all $t>0$. Their $L^2$ estimate grows exponentially as $t \to +\infty$. In \cite{wu2014global}, also regarding the case $\gamma \in [-2,0)$, the author shows that if $f_0 \in L^p$, for $1 < p < +\infty$, then the solution $f(t,\cdot)$ remains in $L^p$. The growth of the $L^p$ estimate is exponential for $\gamma \in (-2,0)$ and double exponential for $\gamma = -2$. The case $\gamma = -2$ also requires an extra moment assumption in the initial data. At the moment, there is no a priori estimate available for the case $\gamma < -2$ (\emph{very soft potentials}) without further assumptions on the function $f$ or the initial data $f_0$ (for example a smallness assumptions). This is also a limitation of our result.

We obtain the following a priori estimate for classical solutions to the Landau equation. It is derived as a consequence of Theorem \ref{t:local-max-pple-intro}. Our result provides an estimate for $\|f(t,\cdot)\|_{L^\infty}$ which does not deteriorate in time and depends only on physically meaningful quantities. In particular, it does not depend on $\|f_0\|_{L^p}$ for any $p>1$.

\begin{thm} \label{t:landau-intro}
Let $\gamma \in [-2,0]$ and $f$ be a solution to the Landau equation \eqref{e:intro-landau} with initial data $f(0,v) = f_0(v)$. Assume that
\begin{align*}
M_1 \leq \int_{\R^d} f_0(v) \dd v &\leq M_0, \\
\int_{\R^d} |v|^2 f_0(v) \dd v &\leq E_0, \\
\int_{\R^d} f_0(v) \log f_0 \dd v &\leq H_0.
\end{align*}
Then, we have the following a priori upper bound on $f$,
\[ f(t,v) \leq \begin{cases}
K t^{-d/2} & \text{ for $t \leq T$},\\
K T^{-d/2} & \text{ for $t \leq T$},
\end{cases}\]
where $K$ and $T$ depend on $d$, $M_1$, $M_0$, $E_0$ and $H_0$ only.
\end{thm}

Note that our estimate does not depend on an upper bound of the initial data $f_0$. We only make the physically meaningful assumptions of boundedness of mass, energy and entropy for the initial data. In this sense, our result is a regularization because we obtain an $L^\infty$ estimate of $f(t,\cdot)$ which does not depend on the $L^\infty$ norm of $f_0$. This is a contrast with the previous results in \cite{alexandre2013some} and \cite{wu2014global} which show a propagation of estimates. A result that is perhaps closer to ours is given in \cite{guillenPersonalcommunication}, but based on completely different methods. There, the authors find a local $L^\infty$ estimate using the DeGiorgi iteration.

Theorem \ref{t:landau-intro} is the local version of Theorem 1.2 in \cite{silvestre2014new} concerning the Boltzmann equation. Note that the result in this article is limited to the space homogeneous case, whereas  the result in \cite{silvestre2014new} holds in the space in-homogeneous case as well. The reason for this difference is that we apply techniques from integro-differential equations to the Boltzmann equation and it is easier in that context to relate the $L^\infty$ norm of the solution with its $L^1$ norm (see Lemma 7.2 in \cite{silvestre2014new} and the paragraph right before it).

Regarding the case $\gamma < -2$ (\emph{very soft potentials}) we only obtain a conditional estimate depending on the norm of $f$ in  $L^\infty_t L^p(\R^d)$ for some $p > d/(d+2+\gamma)$.

After the $L^\infty$ bound in 
Theorem \ref{t:local-max-pple-intro}, it is easy to obtain higher regularity estimates. Indeed, applying the Krylo-Safonov theorem \cite{krylov1980certain} (or alternatively, the theorem of DeGiorgi, Nash and Moser), we immediately obtain a H\"older continuity estimate in both $v$ and $t$. Therefore, the classical Schauder estimates imply H\"older continuity estimates for $D^2_v f$ and $f_t$. These estimates are local. The same reasoning shows that bounded weak solutions are locally $C^2$ in space and $C^1$ in time. Some further analysis is needed in order to establish by what rate these estimates deteriorate (or improve?) as $|v| \to +\infty$.

Theorem \ref{t:landau-intro} immediately implies the existence of classical solutions when the initial data $f_0$ is sufficiently nice. Indeed, assuming that $f_0 \in L^\infty$, in addition to the hypothesis of Theorem \ref{t:landau-intro}, a bounded solution exists locally in time \cite{peskov}. This solution must be classical due to the reasoning in the previous paragraph. Theorem \ref{t:landau-intro} tells us that this bounded solution does not blow up in finite time and is, therefore, a global smooth solution.  In this sense, the control of the $L^\infty$ norm of $f$ is the key to ensure that a global smooth solution exists. This a priori estimate in $L^\infty$ is the focus of this article.

Our result in Theorem \ref{t:landau-intro} is an a priori estimate for classical solutions. We require $f$ to be $C^2$ in space, $C^1$ in time. We do not prove that our $L^\infty$ estimate holds for all weak solutions. This might be relevant when $f_0$ is not smooth enough to guarantee the existence of smooth solutions, or for the conditional $L^\infty$ estimate of Theorem \ref{t:landau-conditional}. The method that we used (based on quantitative maximum principles in the style of the Aleksandrov-Bakelman-Pucci estimate) is not well suited to the usual definition of weak solution in terms of integration by parts against test functions. It might be possible to use an approximation scheme to obtain a solution, for any given initial data, which satisfies our bounds. We did not pursue that issue and we focus  this paper on the a priori estimates only.

\subsection{Notation}

Notation: we use the symbols $\lesssim$ and $\gtrsim$ to denote boundedness up to a constant depending at most on dimension. Likewise, we write $a \approx b$ when $a \gtrsim b$ and $a \lesssim b$.

When we work with general parabolic equations, we use $(t,x)$ as the names for the variables. This is in agreement with the standard literature. When we deal with the Landau equation, we use the names $(t,v)$ instead. In this context, the variable $x$ should be reserved for the space variable in the space inhomogeneous problem. The variables $t$ and $v$ represent time and velocity.

\section{Upper bounds for generic parabolic equations.}
In this section we prove Theorem \ref{t:local-max-pple-intro} together with a version for the borderline case $\alpha = 2p/d$.

The following result is a more precise version of Theorem \ref{t:local-max-pple-intro}. Its proof is, without a doubt, the most difficult part of this paper. The result is related to the local maximum principles for elliptic and parabolic equations in nondivergence form, like the second part of Theorem 4.8 in \cite{caffarelli1995fully} or Proposition 2.4.34 in \cite{imbert2013introduction}. Note that the proofs in those references use some covering argument and the resulting constants are very difficult to compute explicitly. Our proof here is much more direct and thanks to that we are able to obtain an explicit dependence of the upper bound respect to all parameters. Moreover, we can also keep track of the precise required asymptotic bounds on the coefficients at infinity, which is crucial for the application to the Landau equation.

The idea of the proof is somewhat inspired by the Aleksandrov-Bakelman-Pucci estimate or, more specifically, by its parabolic counterpart which was obtained by Krylov \cite{krylov1976sequences}. We do not apply Krylov's estimate here, but we use some of the ideas in its proof. We use a modern approach which consist in estimating the measure of points where $f$ can be touched from above with certain quadratic polynomials (see \cite{savin2007small} and \cite{wu2014global} for examples of the use of this approach to obtain estimates for equations in nondivergence form).

\begin{thm} \label{t:local-max-pple-nonlinear}
Let $f : [0,1] \times \R^d \to \R$ be a bounded continuous function. Let $p \in (1,\infty)$. Assume that wherever $f(t,x)>0$, the function is second differentiable in $x$, differentiable in $t$, and
\[ f_t \leq a_{ij}(t,x) \partial_{ij} f +  C \|f(t,\cdot)\|_{L^\infty}^{1+\alpha}.\]
Here, $ \alpha$ is a parameter in $[0,2p/d)$, the coefficients $a_{ij}$ are locally uniformly elliptic and, for some constants $\delta>0$, $\Lambda>0$, $\kappa \in \R$ and $\beta \geq -\kappa/d$,
\begin{align}
\det (a_{ij}) &\geq \delta (1+|x|)^{\beta d - \kappa}, \label{e:lmpn-delta} \\
(a_{ij}(t,x)) &\leq \Lambda (1+|x|)^{\min(2\beta,2)} \mathrm{I},  \label{e:lmpn-Lambda}  \\
\int_{\R^d} (1+|x|)^\kappa f(t,x)^p \dd x &\leq N. \label{e:lmpn-N}
\end{align}

Let $K$ be the number so that
\begin{equation} \label{e:lmpn-K}
 K^p = C_1 \frac{(1+\Lambda)^{(d+1)} N }{ \delta }.
\end{equation}
Then, for all $(t,x) \in [0,T] \times \R^d$, we have
\begin{equation} \label{e:lmpn-estimate}
 f(t,x) < K t^{-\frac d {2p}},
\end{equation}
where
\begin{equation} \label{e:lmpn-T}
 T = \left( \frac{1+ \Lambda}{C K^\alpha} \right)^{1/\left(1-\frac {\alpha d} {2p}\right)}.
\end{equation}
The constant $C_1$ depends on dimension only.
\end{thm}

\begin{proof}
Let us define $m(t) = K t^{-\frac d {2p} }$.

Assume that the inequality \eqref{e:lmpn-estimate} is not true for all $t \in (0,1)$. The inequality would hold for small values of $t$ because $f$ is bounded. There would be some $t_0 \in (0,1)$  so that
\[ \sup_{x \in \R^d} f(t_0,x) = K t_0^{-\frac d {2p}} \qquad \sup_{x \in \R^d} f(t,x) < K t^{-\frac d {2p}} \ \text{ for all } t<t_0.\]

In particular, there is some point $x_0 \in \R^d$ so that
\begin{align*}
f(t_0,x_0) &> \frac {99} {100} m(t_0) = \frac {99}{100} K t_0^{-\frac d {2p}}, \\
f(t,x) &\leq m(t) \qquad \qquad \qquad \qquad \text{ for all } t<t_0, x \in \R^d.
\end{align*}
Let $r_0 = \sqrt{t_0} (1+|x_0|)^{\min(\beta,1)}$. For any choice of the parameters $h \in [4/5,9/10]$ and $y \in B_{r_0}(x_0)$, we construct the auxiliary functions
\begin{align*}
U(t,x) &= 1 + \frac{t_0 |x|^2} {10 t \, r_0^2} = 1 + \frac{2 |x|^2} {5 t (1+|x_0|)^{2\beta}}, \\
\varphi(t,x) &= h \, m(t) \, U(t,x-y).
\end{align*}

Since $f(t_0,x_0) > 99/100 \, m(t_0)$, $h<9/10$, and $|y-x_0| \leq r_0$, we have $f(t_0,x_0) > \varphi(t_0,x_0)$. Moreover, since $\varphi$ goes to $+\infty$ as $t \to 0$, then there must be a first point $(t,x)$ where $f$ and $\varphi$ cross. That means
\begin{align}
f(t,x) &= \varphi(t,x), \label{e:lmpn-value} \\
f(s,z) &\leq \varphi(s,z) \text{ for all } s\leq t, z \in \R^d.
\end{align}

Such crossing point $(t,x)$ exists for all choices of $y \in B_{r_0} (x_0)$ and $h \in (4/5,9/10)$.

In particular, we have the elementary relations
\begin{align}
\grad f (t,x) &= \grad \varphi(t,x), \label{e:lmpn-grad}\\
f_t(t,x) &\geq \varphi_t(t,x), \label{e:lmpn-ft}\\
D^2_x f(t,x) &\leq D^2_x \varphi (t,x). \label{e:lmpn-D2}
\end{align}

Since $f(t,x) \leq m(t)$ for all $t\leq t_0$, we always have $h > 4/5$, and $\varphi(t,x) \geq 4m(t)/5$ for any value of $(t,x)$, then we deduce
\begin{align*}
h \left( 1 + \frac{t_0 |x-y|^2}{10 t\, r_0^2} \right) &< 1, \\
\frac{t_0 |x-y|^2}{10 t\, r_0^2} &< 1/4, \\
f(t,x) &\geq 4m(t)/5, \\
|x-x_0| &\leq |x-y| + |y-x_0| < 6 r_0, \\
r_0 &\leq  (1+|x_0|) &&\text{from the construction of $r_0$}.
\end{align*}

Let $A \subset (0,t_0) \times \R^d$ be the set of point $(t,x)$ which are crossing points for some value of the parameters $h \in (4/5,9/10)$ and $y \in B_{r_0}(x_0)$. It is possible that for one value of $(h,y)$ we may have two different crossing points $(t,x)$. However, every crossing point $(t,x) \in A$ corresponds to only one choice of $(h,A)$. Indeed, they can be recovered using the relations \eqref{e:lmpn-value} and \eqref{e:lmpn-grad}. Since $f$ is smooth at all crossing points in $A$, it is not hard to see that the map $(t,x) \mapsto (h,y)$ is at least Lipschitz. In fact, we will compute some of its derivatives below. Understanding this map, from the crossing point $(t,x)$ to the value of the parameters $(h,y)$ is the key of this proof.

From now on we consider $h$ and $y$ as functions of $(t,x) \in A$. We write $h=h(t,x)$ and $y = y(t,x)$. We write $\partial(h,y) / \partial(t,x)$ to denote the $R^{(d+1) \times (d+1)}$ matrix valued function corresponding to the derivative of the map $(t,x) \mapsto (h,y)$.

 Since every value of $h \in (4/5,9/10)$ and $y \in B_{r_0}(x_0)$ corresponds to some point $(t,x) \in A$, then the image of this map is the whole cylinder $(4/5,9/10) \times B_{r_0}(x_0)$. Using the elementary Jacobian formula, we get
\begin{equation} \label{e:lmpn-Jacobian}
\frac{ |B_{r_0}| } {10} \approx r_0^d \lesssim \int_A \abs{ \det \frac {\partial (h,y)} {\partial (t,x)} } \dd x \dd t.
\end{equation}

In order to estimate the Jacobian inside the integral, we differentiate the indentities \eqref{e:lmpn-value} and \eqref{e:lmpn-grad}. We obtain
\begin{align*}
f_t &= \varphi_t + m(t) \left( \frac{\partial h}{\partial t} U(x-y) - h \, \grad U \cdot \frac{\partial y}{\partial t} \right), \\
\grad f&= \grad \varphi + m(t) \left( U(x-y) \frac{\partial h}{\partial x} - h \, \grad U \frac{\partial y}{\partial x} \right),\\
D^2 f &= D^2 \varphi \left( \mathrm I - \frac {\partial y}{\partial x} \right) + m(t) \grad U(x-y) \otimes \frac{\partial h}{\partial x}.
\end{align*}
Recalling that $\grad f(t,x)= \grad \varphi(t,x)$, we rewrite the identities above in matrix form.
\[ \begin{pmatrix}
m(t) U(x-y) & -m(t) \, h\, \grad U(x-y) \\
-m(t) \grad U(x-y)^T & m(t) \, h \, D^2 U(x-y)
\end{pmatrix} \cdot 
 \begin{pmatrix}
\frac{\partial h}{\partial t} & \frac{\partial h}{\partial x} \\
\frac{\partial y}{\partial t}  & \frac{\partial y}{\partial x} 
\end{pmatrix} 
=  \begin{pmatrix}
f_t - \varphi_t & 0 \\
? & D^2 \varphi - D^2 f
\end{pmatrix} \]

Here $\grad U \in \R^{1 \times d}$ and $D^2 U \in \R^{d \times d}$. The matrices involved in the identity above are in $\R^{(d+1) \times (d+1)}$. The question mark stands for a value that we did not compute because it is irrelevant for the estimate in this proof (it does not affect the determinant of the right-hand side).

The second factor on the left-hand side is exactly $\partial(t,x)/\partial (h,y)$. We will estimate the determinant of the right hand side using the equation. The first factor depends on our special construction. Taking determinants, we get
\[ \abs{ \det \frac {\partial (h,y)} {\partial (t,x)} } = \frac{ (u_t - \varphi_t) \det(D^2 \varphi - D^2 u) } { m(t)^{d+1} h^d \det \begin{pmatrix} U & -\grad U \\ -\grad U & D^2 U \end{pmatrix} }.\]

Given our choice of the function $U$, we have
\begin{equation} \label{e:lmpn-funnymatrix}
\det  \begin{pmatrix} U & -\grad U \\ 
-\grad U & D^2 U 
\end{pmatrix}  = \left( \frac {t_0} {5 t\, r_0^2} \right)^d \left( 1 - \frac {t_0 |x-y|^2} {10 t r_0^2} \right) \gtrsim t^{-d} r_0^{-2d} t_0^d. 
\end{equation}
The last inequality holds because, as we mentioned before, $t_0 |x-y|^2/(10 t r_0^2) < 1/4$. Note that we also have $|D^2 \varphi| \lesssim t_0 m(t)/(t r_0^2)$.

Since $h \in (4/5,9/10)$, the factor $h^d$ is also bounded below and above depending on $d$ only. We can simplify our estimate of the Jacobian to
\[ \abs{ \det \frac {\partial (h,y)} {\partial (t,x)} } \lesssim \frac{ (u_t - \varphi_t) \det(D^2 \varphi - D^2 u) r_0^{2d} t^d } { m(t)^{d+1} t_0^d }.\]

Let us analyze $\varphi_t$. We have
\begin{align*} 
0 \geq \varphi_t(t,x) &= m'(t) \, h \, U(x-y) - m(t) h \frac{U(t,y-x)} t, \\ 
 &\gtrsim  -\frac {m(t)} t && \text{(recall that we always have $U(t,y-x) < 1/4$)}.
\end{align*}

Recall from \eqref{e:lmpn-ft} and \eqref{e:lmpn-D2} that $u_t - \varphi_t \geq 0$ and $D^2 \varphi - D^2 u \geq 0$. We use the equation in order to estimate the numerator.
\begin{align*} 
 (f_t - \varphi_t) + a_{ij} \left(\partial_{ij} \varphi - \partial_{ij} f \right) &\leq C m(t)^\alpha f - \varphi_t + |a_{ij}| |D^2 \varphi|, \\
&\lesssim m(t) \left( C m(t)^\alpha +  \frac 1 t \left( 1 + t_0 \frac {|a_{ij}|} {r_0^2} \right) \right), \\
&\lesssim m(t) \left( C m(t)^\alpha +  \frac 1 t \left( 1 + \Lambda \right) \right). 
\end{align*}
In the last inequality, we used that \eqref{e:lmpn-Lambda} and the choice of $r_0$ imply that $t_0 |a_{ij}| / r_0^2 \leq \Lambda$.

There is an elementary inequality from linear algebra that says that for any two positive matrices $(a_{ij})$ and $(b_{ij})$ in $\R^{d \times d}$,
\[ a_{ij} b_{ij} \geq d \det (a_{ij})^{1/d} \det (b_{ij})^{1/d}. \]
Applying this inequality to $(a_{ij})$ and $(D^2 \varphi - D^2 f)$, we get
\[ \det(D^2 \varphi - D^2 f) \leq \frac{m(t)^d \left( C m(t)^\alpha +  \frac 1 t(1+\Lambda) \right)^d} {\det(a_{ij})}.\]

Therefore, we have
\begin{align*} 
 \abs{ \det \frac {\partial (h,y)} {\partial (t,x)} } &\lesssim \frac{ m(t)^{d+1} \left( C m(t)^\alpha +  \frac 1 t(1+\Lambda) \right)^{d+1}  r_0^{2d} t^d } { \det(a_{ij})  m(t)^{d+1} t_0^{d} }, \\
&\lesssim \frac  {r_0^{d} (1+|x_0|)^\kappa } { \delta t_0^{d/2}} \left( C^{(d+1)} m(t)^{(d+1)\alpha} t^d + t^{-1} (1+\Lambda)^{d+1}   \right), \\
\end{align*}
For the last inequality we used that \eqref{e:lmpn-delta} and our definition of $r_0$ imply that $r_0^d / \det(a_{ij}) \leq t_0^{d/2} (1+|x_0|)^\kappa$.

Recalling \eqref{e:lmpn-Jacobian},
\begin{equation} \label{e:lmpn-ineq-to-contradict}
1 \lesssim \int_A  \frac {(1+|x_0|)^\kappa} {\delta t_0^{d/2}} \left( C^{(d+1)} m(t)^{(d+1)\alpha} t^d + t^{-1} (1+\Lambda)^{d+1} \right) \dd x \dd t.
\end{equation}

Recall that  $4m(t)/5 \leq f(t,x) \leq m(t)$ for all $(t,x) \in A$. Moreover, $r_0 \leq (1+|x_0|)$. Because of our assumption \eqref{e:lmpn-N}, for any $(t,x) \in A$, we have
\begin{equation} \label{e:lmpn-using-kappa}
\int_{B_{r_0}(x_0)} f(t,z)^p \dd z \lesssim N (1+|x_0|)^{-\kappa}.
\end{equation}

Let $A(t) = \{ z : (t,z) \in A\}$. Since $f \gtrsim m(t)$ on $A(t)$, $|A(t)| \leq N (1+|x_0|)^{-\kappa} / m(t)^p$ because of Chebyshev's inequality. Therefore, we deduce from \eqref{e:lmpn-ineq-to-contradict} that
\begin{align*} 
1 &\lesssim \int_0^{t_0} \frac {N} {m(t)^p \, \delta  t_0^{-d/2}}  \left( C^{(d+1)} m(t)^{(d+1)\alpha} t^d + t^{-1} (1+\Lambda)^{d+1} \right) \dd t, \\ &\lesssim \frac {N} {K^p \delta t_0^{d/2}} \int_0^{t_0} \left( C^{(d+1)} m(t)^{(d+1) \alpha} t^{3d/2} + ( 1 + \Lambda )^{d+1} t^{d/2-1} \right) \dd t
\end{align*}

Integrating, we get the inequality,
\begin{equation} \label{e:lmpn-specifictocontradict}
 K^p \lesssim \frac N{ \delta } \left( C^{(d+1)} K^{\alpha(d+1)} t_0^{(d+1)\left(1-\frac{\alpha d} {2 p}\right)  } + ( 1 + \Lambda )^{d+1}  \right).
\end{equation}

We arrive to a contradiction when $K$ is sufficiently large so that $K^p \gtrsim \frac N \delta (1+\Lambda)^{d+1}$, and $T$ is sufficiently small (depending on $K$) so that the second term in the right-hand side of the inequality is larger than the first.

Recall that $\alpha < 2p/d$. This is a contradiction with our choice of $K$ and $T$ in \eqref{e:lmpn-K} and \eqref{e:lmpn-T}. Indeed, the first term in the right-hand side is
\begin{align*}
\frac N { \delta}  C^{(d+1)} K^{\alpha(d+1)} t_0^{(d+1)\left(1-\frac{\alpha d} {2 p}\right)  } &< \frac N \delta  C^{(d+1)} K^{\alpha(d+1)} T^{(d+1)\left(1-\frac{\alpha d} {2 p}\right)  },\\
& \leq \frac N { \delta} ( 1 + \Lambda )^{d+1}.
\end{align*}

\end{proof}

%\begin{remark}
%The restriction $r_0 \leq 1+|x_0|$ is necessary to catch the decay in $f$ in \eqref{e:lmpn-using-kappa}. This would not be an issue if $\kappa \leq 0$. In that case, the theorem may hold for a wider selection of parameters.
%\end{remark}

As we will see later, the range of $\alpha$ in Theorem \ref{t:local-max-pple-nonlinear} is unimprovable with the given hypothesis. 

The following Theorem is a version of the bordeline case $\alpha = 2/d$ and $p=1$ with an $\log$-improved hypothesis on the nonlinearity.
 
\begin{thm} \label{t:local-max-pple-nonlinear-log}
Let $f : [0,1] \times \R^d \to \R$ be a bounded continuous function. Assume that wherever $f(t,x)>0$, the function is second differentiable in $x$, differentiable in $t$, and
\[ f_t \leq a_{ij}(t,x) \partial_{ij} f +  C \frac{\|f(t,\cdot)\|_{L^\infty}^{2/d} }{\log \left(1+ \|f(t,\cdot)\|_{L^\infty}\right)^\eps} \, f.\]
Here, $\eps>0$, the coefficients $a_{ij}$ are locally uniformly elliptic and, for some constants $\delta>0$, $\Lambda>0$, $\beta \in \R$, and $\kappa > -\beta d$,
\begin{align}
\det (a_{ij}) &\geq \delta (1+|x|)^{\beta d - \kappa}, \label{e:lmpn-delta-log} \\
(a_{ij}(t,x)) &\leq \Lambda (1+|x|)^{\min(2\beta,2)} \mathrm{I},  \label{e:lmpn-Lambda-log}  \\
\int_{\R^d} (1+|x|)^\kappa f(t,x) \dd x &\leq N. \label{e:lmpn-N-log}
\end{align}
Let $K$ be the positive number so that
\begin{equation} \label{e:lmpn-K-log}
 K = C_1 \frac{ N \left(  1+ \Lambda \right)^{(d+1)} }{ \delta }.
\end{equation}
Then, for all $(t,x) \in [0,T] \times \R^d$, we have
\begin{equation} \label{e:lmpn-estimate-log}
 f(t,x) < K t^{-d/2},
\end{equation}
where $T$ is the positive number given by
\begin{equation} \label{e:lmpn-T-log}
T = K^{2/d} \exp\left( -\frac 2 d \left( \frac{C_2 K^{2/d}}{1+\Lambda} \right)^{1/\eps} \right)
\end{equation}
The constants $C_1$ and $C_2$ depend on dimension only.
\end{thm}

The proof of Theorem \ref{t:local-max-pple-nonlinear-log} is essentially the same as the proof of Theorem \ref{t:local-max-pple-nonlinear}. Below, we explain the small modification that we need.

\begin{proof}[Sketch proof]
The proof is identical to the one of Theorem \ref{t:local-max-pple-nonlinear}  with $\alpha = 2/d$ and $p=1$, replacing every occurrence of $m(t)^\alpha$ with $m(t)^{2/d} / \log(1+m(t))^\eps$.

Instead of \eqref{e:lmpn-specifictocontradict}, we get the inequality
\[ K \lesssim \frac N \delta \left( C^{d+1} \frac{K^{2+2/d}}{\log(1+K T^{-d/2})^{\eps(d+1)}} + (1+\Lambda)^{d+1} \right).\]

Our choice of $K$ and $T$ consists in picking $K$ large so that the left hand side is larger than twice the second term in the right hand side. Then we pick $T$ small, so that to make the first term in the right hand side smaller than the second. With these choices, explicitly given in \eqref{e:lmpn-K-log} and \eqref{e:lmpn-T-log}, we arrive to a contradiction and finish the proof.
\end{proof}

\begin{cor} \label{c:local-max-pple}
Let $f$, $p$, $K$ and $T$ be as in Theorem \ref{t:local-max-pple-nonlinear} or Theorem \ref{t:local-max-pple-nonlinear-log}. Then, we have the estimate
\[ f(t,x) < \begin{cases}
Kt^{-d/(2p)} & \text{ if $t \leq T$}, \\
KT^{-d/(2p)} & \text{ if $t \geq T$}.
\end{cases}.\]
\end{cor}

\begin{proof}
The bound for $t \leq T$ is the result of Theorems \ref{t:local-max-pple-nonlinear} and \ref{t:local-max-pple-nonlinear-log}. For $t > T$, we apply the same theorems to the translated function
\[ \tilde f(s,x) = f(s+t-T,x).\]
\end{proof}

\begin{remark}
Note that the proof of Theorems \ref{t:local-max-pple-nonlinear} and \ref{t:local-max-pple-nonlinear-log} are nonvariational. They do not necessarily apply to solutions of the equation in the sense of distributions. The right notion of solution that would make (essentially) the same proof hold is that of \emph{viscosity solution} when the second order term $a_{ij}(t,v) \partial_{ij} f$ is replaced by the extremal operator in terms of the inequalities \eqref{e:lmpn-delta} and \eqref{e:lmpn-Lambda}. This notion of solution may not be an appropriate starting point for the Landau equation.
\end{remark}

\subsection{Unimprovability of the upper bounds}

We show that the value of $p$ in the assumptions of Theorem \ref{t:local-max-pple-nonlinear} is optimal. The next proposition shows that the upper bound cannot hold if $\alpha \geq 2p/d$. In particular, it is only possible to handle the borderline case $\alpha = 2p/d$ making some slightly stronger assumption, like the logarithmic correction of Theorem \ref{t:local-max-pple-nonlinear-log}.

\begin{prop} \label{p:unimprovability}
Let $\alpha \geq 2p/d$. There exists a function $f: [0,1) \times \R^d \to [0,\infty)$, supported in $[0,1) \times B_1$, so that $\|f(t,\cdot)\|_{L^p}$ is constant in time and
\begin{equation} \label{e:unimp}
 f_t \leq \lap f + f^{1+\alpha},
\end{equation}
however $f(t,0) \to +\infty$ as $t \to 1$.
\end{prop}

\begin{proof}
Let $\varphi : \R^d \to [0,\infty)$ be a smooth, nonnegative function. We choose $\varphi$ radially symmetric and non-increasing along rays. Let us also choose it so that $\supp \varphi = \overline B_1$ and note that in this case we have, for any unit vector $e$,
\begin{equation} \label{e:unimprov-lap-wis}
  \lim_{r \to 1} \frac{\lap \varphi(re)}{|\grad \varphi (re)| } = \lim_{r \to 1} \frac{\lap \varphi(re)}{\varphi (re) } = +\infty.
\end{equation}

We set
\[ f(t,x) = (1-t)^{-\frac d {2p}} \varphi \left( \frac{x} {(1-t)^{1/2}} \right).\]
Thus,
\begin{align*}
f_t(t,x) &= \frac 1 {1-t} \left( \frac d {2p} \varphi \left( \frac{x} {(1-t)^{1/2}} \right) + \frac 12 \left( \frac{x} {(1-t)^{1/2}} \right) \cdot \grad \varphi \left( \frac{x} {(1-t)^{1/2}} \right) \right), \\
\lap f(t,x) &= \frac 1 {1-t} \lap \varphi\left( \frac{x} {(1-t)^{1/2}} \right), \\
f^{1+\alpha} &= (1-t)^{-\frac {d\alpha}{2p}} \varphi\left( \frac{x} {(1-t)^{1/2}} \right)^{1+\alpha}
\end{align*}
We see that \eqref{e:unimp} is satisfied provided that $d \alpha / (2p) \geq 1$ and 
\begin{equation} \label{e:unimprov-statioary}
 \frac d {2p} \varphi(y) + \frac 12 y \cdot \grad \varphi(y) \leq \lap \varphi(y) + \varphi(y)^{1+\alpha}.
\end{equation}

Note that we used the assumption $\alpha > 2p/d$ here to justify that $(1-t)^{-\frac {d\alpha}{2p}} \geq 1$. This is the only place where that assumption is used.

From \eqref{e:unimprov-lap-wis}, we see that the term $\lap \varphi(y)$ is larger than the left hand side for $y \in B_1 \setminus B_{1-\delta}$ with $\delta>0$ small enough. Note that in this case we control the left hand side with the first term in the right hand side only. All these terms are linear, so the same inequality holds for any function $k \varphi(y)$ instead of $\varphi(y)$, if $k > 0$.

For $y \in \overline{B_{1-\delta}}$, we have $\varphi(y) \geq \kappa$ for some positive number $\kappa > 0$. If the inequality \eqref{e:unimprov-statioary} did not hold in $\overline{B_{1-\delta}}$, we would replace $\varphi(y)$ with $K \varphi(y)$ for some large constant $K$. The inequality \eqref{e:unimprov-statioary} will now hold in $\overline{B_{1-\delta}(y)}$ as well since the second term of the right hand side is raised to a power larger than one.
\end{proof}

\section{The Landau equation for moderately soft potentials}

We focus on the Landau equation \eqref{e:intro-landau} for the rest of this paper. The solution $f$ is a nonnegative function whose $L^1$ norm and second moment are constant in $t$.  Moreover, its entropy is monotone decreasing. That means
\begin{align*}
M &:= \int_{\R^d} f(t,v) \dd v && \text{ is constant in } t, \\
E &:= \int_{\R^d} |v|^2 f(t,v) \dd v && \text{ is constant in } t, \\
H(t) &:= \int_{\R^d} f(t,v) \log f(t,v) \dd v && \text{ is non-increasing in } t.
\end{align*}

\subsection{Ellipticity estimates for $\gamma < 0$.}

In this section, we show upper and lower bounds for the coefficients $\bar a_{ij}$ and $\bar c$. Throughout this section, we assume that $f$ is a nonnegative function satisfying the following bounds.
\begin{align}
M_1 \leq \int_{\R^d} f(v) \dd v &\leq M_0, \label{e:landau-mass-bound} \\
\int_{\R^d} |v|^2 f(v) \dd v &\leq E_0, \label{e:landau-energy-bound}  \\
\int_{\R^d} f(v) \log f(v) \dd v &\leq H_0. \label{e:landau-entropy-bound} 
\end{align}

It is well known that we can conclude that $f$ is bounded in $L \log L$. That means that there exists a $\tilde H_0$ depending on $M_0$, $E_0$ and $H_0$ so that
\begin{equation} \label{e:landau-entropyprime-bound} 
 \int_{\R^d} f(v) \log(1+f(v)) \dd v \leq \tilde H_0.
\end{equation}

The function $f$ solving the Landau equation depends on $t$. In this section the time dependence is irrelevant. These computations hold for every fixed $t$. Thus, we omit the variable $t$ from the formulas.

The following two lemmas provide upper and lower bounds for $\bar a_{ij}$. They follow from standard computations, which probably appeared first in \cite{desvillettes2000spatially} for the case of hard potentials $\gamma \geq 0$. The following two lemmas are thus just an adaptation of the computation in \cite{desvillettes2000spatially} to the case of soft potentials.

\begin{lemma} \label{l:coeffs-moderately-soft-potentials}
Let $f$ satisfy \eqref{e:landau-mass-bound}, \eqref{e:landau-energy-bound} and \eqref{e:landau-entropy-bound}, and $\bar a_{ij}$ be the diffusion coefficients of the Landau equation given in \eqref{e:intro-aij}. Assume $\gamma \in [-2,0)$. The following estimates hold.
\begin{align*}
| \bar a_{ij} | &\leq C (1+|v|^{\gamma+2}), \\
\det \bar a_{ij} &\geq c (1+|v|)^{(d-1)(\gamma + 2) + \gamma},
\end{align*}
where $c$ and $C$ depend on the mass, energy and entropy of $f$.
\end{lemma}

\begin{lemma} \label{l:coeffs-very-soft-potentials}
Let $f$ satisfy \eqref{e:landau-mass-bound}, \eqref{e:landau-energy-bound} and \eqref{e:landau-entropy-bound}, and $\bar a_{ij}$ be the diffusion coefficients of the Landau equation given in \eqref{e:intro-aij}. Assume $-d-2 < \gamma < -2$. The following estimates hold.
\begin{align*}
| \bar a_{ij} | &\leq C |f|_{L^\infty}^{-(\gamma+2)/d}, \\
\det \bar a_{ij} &\geq c (1+|v|)^{(d-1)(\gamma+2) + \gamma},
\end{align*}
where $c$ and $C$ depend on the mass, energy and entropy of $f$.
\end{lemma}

The previous two lemmas are relatively standard. The upper bounds follow easily from the formula \eqref{e:intro-aij} in the same way as the upper bound for $\bar c$ in Lemma \ref{l:c-coeff} below. We provide a complete proof of the lower bound in Lemma \ref{l:coeffs-moderately-soft-potentials} which is the one that is used in this article. The proof of this bound also follows the lines of the computation in \cite{desvillettes2000spatially}.

We need the following auxiliary lemma first. It is the same as Lemma 4.6 in \cite{silvestre2014new}.

\begin{lemma} \label{l:thick-set}
Assume that $f$ satisfies \eqref{e:landau-mass-bound}, \eqref{e:landau-energy-bound} and \eqref{e:landau-entropyprime-bound}. There exists three positive numbers $R>0$, $\ell>0$ and $\mu>0$, depending only on $M_1$, $E_0$, $\tilde H_0$ and the dimension $d$, so that 
\[ |\{ x \in B_R : f(x) \geq \ell \}| \geq \mu.\]
\end{lemma}

\begin{proof}
Let us choose $R$ so that $E_0 / R^2 \leq M_1/2$. In that way, we have
\[ \int_{\R^d \setminus B_R} f(v) \dd v \leq R^{-2} E_0 \leq  \frac {M_1} 2.\]
Therefore, we make sure that 
\[ \int_{B_R} f(v) \dd v \geq \frac {M_1} 2.\]

Let $\ell$ be so that $\ell |B_R| = M_1/4$. Thus,
\[ \int_{B_R \cap \{f \leq \ell\}} f(v) \dd v \leq \ell |B_R| = \frac {M_1} 4.\]

We now use the entropy bound. Let $T>0$ be the number so that $\tilde H_0 / \log(1+T) = M_1/8$. Therefore
\[ \int_{\{f > T\}} f(v) \dd v \leq \frac {\tilde H_0} {\log(1+T)} = \frac {M_1} 8.\]
And thus,
\[ \int_{\{T \geq f \geq \ell\}} f(v) \dd v \geq \frac {M_1} 8\]

Finally, 
\[ |\{f \geq \ell\}| \geq |\{T \geq f \geq \ell\}| \geq \frac {M_1}{8T} =: \mu.\]
\end{proof}

\begin{proof}[Proof of the lower bounds in Lemmas \ref{l:coeffs-moderately-soft-potentials} and \ref{l:coeffs-very-soft-potentials}]
The proofs of the lower bound for $\det \bar a_{ij}$ follows by analyzing the expression \eqref{e:intro-aij} in light of Lemma \ref{l:thick-set}.

Let $R$, $\ell$ and $\mu$ be as in Lemma \ref{l:thick-set}. Let $e$ be any unit vector. Using the formula \eqref{e:intro-aij}, we get
\begin{align*}
\langle \bar a_{ij} e, e \rangle &= \int_{\R^d} \left( 1 - \left( \frac{w\cdot e}{|w|} \right)^2 \right) |w|^{\gamma+2} f(v-w) \dd w, \\
&\geq \ell \int_{\{v+w \in B_R\} \cap \{f(v+w) \geq \ell\}} \left( 1 - \left( \frac{w\cdot e}{|w|} \right)^2 \right) |w|^{\gamma+2} \dd w.
\end{align*}
In order to get a good lower bound for the left-hand side, we naturally need to avoid the set of points $w$ which are aligned to $e$ so that $w \cdot e / |w|$ is close to one.

Depending on $v$, $R$ and $\mu$, we can always find an $\eps_0 = \eps_0(|v|,R,\mu) > 0$ so that
\[ \left\vert \left\{ w : v+w \in B_R \text{ and }  \left( \frac{w\cdot e}{|w|} \right)^2 > 1-\eps_0 \right\} \right\vert < \mu / 2. \]
From this, we infer that
\begin{equation} \label{e:lowerboundaij-1}
 \langle \bar a_{ij} e, e \rangle \geq \frac \mu 2 \ell \eps_0 |v+R|^{\gamma+2} \geq c \eps_0 (1+|v|)^{\gamma+2}.
\end{equation}
This already establishes the lower bound for small values of $|v|$. The key of the lemma is to analyze how $\eps_0$ depends on $|v|$. That depends on the orientation of the vector $e$ respect to $v$. For the remaining of the proof, let us assume that $|v| > 2R$.

Let us consider first the case that $v \cdot e = 0$. In this case, for every $w$ so that $v+w \in B_R$, since $|v| > 2R$, we have
\[ \left( \frac{ w \cdot e } {|w|} \right)^2 \leq \frac 1 5.\]
So, when $e \cdot v = 0$, we can choose $\eps_0 = 4/5$ independently of $R$.

Let us now consider the case that $e$ is not perpendicular to $v$. In this case, the set
\[ \left \{ w : \left( \frac{w\cdot e}{|w|} \right)^2 > 1-\eps_0 \right\} \]
is a round cone centered at the origin, with opening angle $\approx \sqrt{\eps_0}$. Its intersection with $\{ w: v+w \in B_R \}$ is contained in a cylinder of width $\lesssim \sqrt{\eps_0} |v|$ which cut $B_R$ is some direction. The worst case scenario is when $e = v / |v|$, in which case the cylinder cuts $B_R$ diametrically. Therefore
\[ \left \{ w : v+w \in B_R \text{ and } \left( \frac{w\cdot e}{|w|} \right)^2 > 1-\eps_0 \right\} \lesssim R \left( \sqrt{\eps_0} |v| \right)^{d-1}, \]
and we must choose $\eps_0 \approx \left( \frac \mu {2R} \right)^{2/(d-1)} |v|^{-2}$.

Going back to \eqref{e:lowerboundaij-1}, and summarizing, we obtained,
\[ \langle \bar a_{ij} e, e \rangle \geq c \begin{cases}
(1+|v|)^{\gamma+2} & \text{ if } v \cdot e = 0, \\
(1+|v|)^{\gamma} & \text{ for any } v.
\end{cases}\]
In order to derive a lower bound for $\det \{\bar a_{ij}\}$, we recall the min-max formulas for eigenvalues of symmetric matrices.
\[ \lambda_k = \min_{\dim S = k} \left( \max_{\substack{e \in S \\ |e|=1 }} \, \langle \bar a_{ij} e,e\rangle \right).\]
For $k=1$, all we can say is that $\lambda_1 \geq c(1+|v|)^{\gamma}$. For $k \geq 2$, any subspace $S \subset \R^d$ will contain some unit vector $e$ perpendicular to $v$, therefore $\lambda_k \geq c(1+|v|)^{\gamma+2}$. Since $\det \bar a_{ij} = \lambda_1 \lambda_2 \dots \lambda_d$, the conclusion of the lemma follows.
\end{proof}

\begin{lemma} \label{l:c-coeff}
Let $f$ satisfy \eqref{e:landau-mass-bound}. For the function $c(t,v)$ given in \eqref{e:intro-c}, we have the following estimate for any $\gamma \in [-d,0]$.
\[ \bar c(t,v) \lesssim M_0^{1+\gamma/d} |f|_{L^\infty}^{-\gamma/d}.\]
\end{lemma}

\begin{proof}
The estimate is trivial in the borderline case $\gamma=-d$. For $\gamma \in (-d,0)$, we use the formula for $\bar c(t,v)$,
\begin{align*} 
 c(t,v) &= c_{d,\gamma} \int_{\R^d} |w|^\gamma f(v-w) \dd w, \\
&=  c_{d,\gamma} \int_{B_R} |w|^\gamma f(v-w) \dd w + c_{d,\gamma} \int_{\R^d \setminus B_R} |w|^\gamma f(v-w) \dd w, \\
&\leq C \left( R^{d+\gamma} \|f\|_{L^\infty} + R^\gamma  M_0 \right).
\end{align*}
The inequality above holds for any value of $R>0$. We conclude the proof picking $R =\left(  M_0  / \|f\|_{L^\infty} \right)^{1/d}$.
\end{proof}

\begin{lemma} \label{l:c-coeff-log}
Let $f$ satisfy \eqref{e:landau-mass-bound}, \eqref{e:landau-energy-bound} and \eqref{e:landau-entropy-bound}. Let $\bar c(t,v)$ be the function given in \eqref{e:intro-c}. 

We have the following estimate for any $\gamma \in (-d,0)$,
\[ \bar c(t,v) \leq C |f|_{L^\infty}^{-\gamma/d} \log(1+\|f\|_{L^\infty} )^{(d+\gamma) / (2\gamma)} , \]
where $C$ depends on $d$, $M_0$, $E_0$ and $H_0$ only.

(Note that the logarithmic term is raised to a negative power because $\gamma \in (-d,0)$)
\end{lemma}

\begin{proof}
Given the result of Lemma \ref{l:c-coeff}, we are only left to prove the case of $\|f\|_{L^\infty}$ large.

Recall that we have a bounded norm for $f$ in $L \log L$ depending on $M_0$, $E_0$ and $H_0$. In particular, \eqref{e:landau-entropyprime-bound} holds.

We use the expression for $\bar c(t,v)$,
\[ \bar c(t,v) = c_{d,\gamma} \int_{\R^d} f(v-w) |w|^\gamma \dd w.\]

This time we split the integral into three subregions: $B_r$, $B_R \setminus B_r$ and $\R^d \setminus B_R$, for $R > r > 0$ to be given by
\begin{align*}
r &= \|f\|_{L^\infty}^{-1/d} \log \left( 1 + \|f\|_{L^\infty} \right)^{-1/d}, \\
R &= \|f\|_{L^\infty}^{-1/d} \log \left( 1 + \|f\|_{L^\infty} \right)^{-(d+\gamma)/(\gamma d)}.
\end{align*}

For the innermost and outermost domains, we have a simple estimate.
\begin{align} 
\int_{B_r} f(v-w) |w|^\gamma \dd w
&\lesssim \|f\|_{L^\infty} r^{d+\gamma} = \|f\|^{-\gamma/d} \log(1+\|f\|_{L^\infty})^{-(d+\gamma)/d}, \\
\int_{\R^d \setminus B_R} f(v-w) |w|^\gamma \dd w
&\lesssim M_0 R^{\gamma}= M_0 \|f\|^{-\gamma/d} \log(1+\|f\|_{L^\infty})^{-(d+\gamma)/d}.
\end{align}

In order to estimate the middle region, we set 
\[ \lambda = \|f\|_{L^\infty} \log(1+\|f\|_{L^\infty})^{(d+\gamma)/\gamma}. \] 
Note that $(d+\gamma)/\gamma<0$. Since we only need to do this proof when $|f\|_{L^\infty}$ is sufficiently large, we can assume $\lambda \leq \|f\|_{L^\infty}$ and
\[ \log(1+\lambda) \geq \frac 12 \log(1+\|f\|_{L^\infty}).\]

Therefore,
\begin{align*} 
 \int_{B_R \setminus B_r} f(v-w) |w|^\gamma \dd w &= \int_{(B_R \setminus B_r) \cap \{f > \lambda\}} f(v-w) |w|^\gamma \dd w + \int_{(B_R \setminus B_r) \cap \{f \leq \lambda\}} f(v-w) |w|^\gamma \dd w, \\
&\leq \frac {\tilde H_0} {\log(1+\lambda)} r^\gamma + \lambda R^{d+\gamma}, \\
&= (2 N+1) \|f\|_{L^\infty}^{-\gamma/d} \log(1+\|f\|_{L^\infty})^{-(d+\gamma)/d}   
\end{align*}
We found the right bound for each of the three terms and thus we finished the proof.
\end{proof}

Finally, in order to obtain conditional upper bounds in the case $\gamma < -2$, we need estimates for $\bar a_{ij}$ and $\bar c$ in terms of $\|f\|_{L^p_\kappa}$.

\begin{lemma} \label{l:aij-coef-Lp-very-soft}
Let $f$ be a function in $L^p_\kappa$. That is
\[ \|f\|_{L^p_\kappa} = \left( \int_{\R^d} (1+|v|)^\kappa |f(v)|^p \right)^{1/p} < +\infty.\]
Assume $-d \leq \gamma<-2$, $d/(d+\gamma)> p > d/ (d+2+\gamma)$ and $\kappa > p(2+\gamma+d)-d$.

Let $\bar a_{ij}$ and $\bar c$ be the functions given in \eqref{e:intro-aij} and \eqref{e:intro-c}. Then
\begin{align*}
\bar c &\leq C \|f\|_{L^p}^{\frac{p(d+\gamma)} {d} } \|f\|_{L^\infty}^{1-\frac{p(d+\gamma)} {d}},\\
|\bar a_{ij}| &\leq C \|f\|_{L^p_\kappa} (1+|v|)^{2+\gamma+d - d/p}.
\end{align*}
for some constant $C$ depending on $p$, $\kappa$, $d$ and $\gamma$.
\end{lemma}

\begin{proof}
The upper bound for $\bar c$ is proved in the same way as Lemma \ref{l:c-coeff} but using the $L^p$ norm of $f$ instead of its $L^1$ norm.

Let us concentrate in the proof of the upper bound for $|\bar a_{ij}|$. 

Let $R = 2(|v|+1)$. We split the integral in \eqref{e:intro-aij} into subdomains as follows.
\begin{align*}
|\bar a_{ij}| & \leq a_{d,\gamma} \int_{\R^d} |w|^{\gamma+2} f(v-w) \dd w, \\
&\leq a_{d,\gamma} \left( \int_{B_R} |w|^{\gamma+2} f(v-w) \dd w + \sum_{k=0}^\infty \int_{B_{2^{k+1} R} \setminus B_{2^k R}} |w|^{\gamma+2} f(v-w) \dd w \right),\\
\intertext{Applying the H\"older's inequality in each integral}
&\leq C \left( \|f\|_{L^p(B_R(v))} R^{2+\gamma+d-d/p} + \sum_{k=0}^\infty 
\|f\|_{L^p(B_{2^{k+1} R}(v) \setminus B_{2^k R}(v))} (2^k R)^{\gamma+2+d-d/p} 
\right), \\
&\leq C R^{2+\gamma+d-d/p} \left( \|f\|_{L^p}  + R^{-\kappa/p} \sum_{k=0}^\infty 
\|f\|_{L^p_\kappa} (2^k)^{\gamma+2+d-d/p-\kappa/p} 
\right), \\
&\leq C R^{2+\gamma+d-d/p} \|f\|_{L^p_\kappa} \qquad \qquad \text{provided that } \gamma+2+d-d/p - \kappa/p < 0.
\end{align*}

\end{proof}

\subsection{Upper bound for the Landau equation with $\gamma \in [-2,0]$}

We state the following result as a theorem to stress its importance. It is actually a corollary of Theorems \ref{t:local-max-pple-nonlinear} and \ref{t:local-max-pple-nonlinear-log}.

\begin{thm} \label{t:landau}
Let $f$ be a solution to the Landau equation \eqref{e:intro-landau} with $\gamma \in [-2,0]$ and initial data $f(0,v) = f_0(v)$. Assume that
\begin{align*}
M_1 \leq \int_{\R^d} f_0(v) \dd v &\leq M_0, \\
\int_{\R^d} |v|^2 f_0(v) \dd v &\leq E_0, \\
\int_{\R^d} f_0(v) \log f_0 \dd v &\leq H_0.
\end{align*}
Then, we have the following a priori upper bound on $f$,
\[ f(t,v) \leq \begin{cases}
K t^{-d/2} & \text{ for $t \leq T$},\\
K T^{-d/2} & \text{ for $t \leq T$}.
\end{cases},\]
where $K$ and $T$ depend on $d$, $M_1$, $M_0$, $E_0$ and $H_0$ only.
\end{thm}

\begin{proof}
Because of the conservation of mass and energy, and the monotonicity of entropy, the following inequalities hold for all $t \geq 0$,
\begin{align*}
M_1 \leq \int_{\R^d} f(t,v) \dd v &\leq M_0, \\
\int_{\R^d} |v|^2 f(t,v) \dd v &\leq E_0, \\
\int_{\R^d} f(t,v) \log f(t,v) \dd v &\leq H_0.
\end{align*}
%It is well known that this implies that
%\[ \int_{\R^d} f(t,v) \log(1+f(t,v)) \dd v \leq \tilde H_0,\]
%for a constant $\tilde H_0$ depending on $H_0$, $M_0$ and $E_0$.

Depending on whether $\gamma > -2$ or $\gamma=-2$ we apply Theorem \ref{t:local-max-pple-nonlinear} or Theorem \ref{t:local-max-pple-nonlinear-log} and Corollary \ref{c:local-max-pple}.

In either case we have, for all $(t,v)$,
\begin{align*}
\det (\bar a_{ij}(t,v)) &\geq \delta (1+|v|)^{(d-1)(\gamma+2)+\gamma}, \\
 |\bar a_{ij}(t,v)| &\leq \Lambda (1+|v|)^{\gamma+2},
\end{align*}
where $\delta$ and $\Lambda$ are given in Lemma \ref{l:coeffs-moderately-soft-potentials}.

From Lemma \ref{l:c-coeff}, we also have
\[ \bar c(t,v) \leq C \|f(t,\cdot)\|_{L^\infty}^{-\gamma/d}.\]

Since $f$ has finite second moments,
\[ \int_{\R^d} (1+|v|)^2 f(v) \dd v \leq 2M_0 + 2E_0 \]

If $\gamma \in (-2,0]$, we apply Theorem \ref{t:local-max-pple-nonlinear} with $\alpha = -\gamma/d$, $p=1$, $\beta=2+\gamma$ and $\kappa = 2$ to conclude.

If $\gamma = -2$ (assuming $d\geq 3$), we resort to Lemma \ref{l:c-coeff-log} to obtain
\[ \bar c(t,v) \leq C \|f(t,\cdot)\|_{L^\infty}^{-\gamma/d} \log (1+ \|f(t,\cdot)\|_{L^\infty})^{-(d-2)/4}.\]
Thus, we can apply Theorem \ref{t:local-max-pple-nonlinear-log} with $\kappa=2$ and $\beta = 0$ to conclude also this borderline case.
\end{proof}

\subsection{Conditional upper bound for the Landau equation with $\gamma < -2$.}

We now consider the case of very soft potentials. The following upper bound is conditional to the assumption that the quantity $\int (1+|v|)^\kappa f(t,v)^p \dd v$ bounded by $W_0$ for all $t \geq 0$ and some appropriately large values of $\kappa$ and $p$.

\begin{thm} \label{t:landau-conditional}
Let $f$ be a solution to the Landau equation \eqref{e:intro-landau} with $\gamma \in [-d,-2)$ and initial data $f(0,v) = f_0(v)$. Assume that
\begin{align*}
M_1 \leq \int_{\R^d} f_0(v) \dd v &\leq M_0, \\
\int_{\R^d} |v|^2 f_0(v) \dd v &\leq E_0, \\
\int_{\R^d} f_0(v) \log f_0 \dd v &\leq H_0, \\
\|f(t,\cdot)\|_{L^p_\kappa}^p = \int_{\R^d} (1+|v|)^\kappa f(t,v)^p \dd v &\leq W_0 && \text{ for all $t \geq 0$},
\end{align*}
for some $p \in (d/(d+2+\gamma), d/(d+\gamma))$ and 
\[\kappa > \max \left(p(2+\gamma+d) - d, 2 + \frac {d^2} 2 \left(1 - \frac 1 p \right) - d \left( 1+\frac \gamma 2 \right) \right).\]

Then, we have the following a priori upper bound on $f$,
\[ f(t,v) \leq \begin{cases}
K t^{-d/2} & \text{ for $t \leq T$},\\
K T^{-d/2} & \text{ for $t \leq T$}.
\end{cases},\]
where $K$ and $T$ depend on $d$, $M_1$, $M_0$, $E_0$, $H_0$ and $W_0$ only.
\end{thm}

\begin{proof}
As before, we have the lower bound for $\bar a_{ij}$ given in Lemma \ref{l:coeffs-very-soft-potentials}.
\[ \det (\bar a_{ij}(t,v)) \geq \delta (1+|v|)^{(d-1)(\gamma+2)+\gamma}.\]

The upper bounds for both $\bar a_{ij}$ and $\bar c$ are obtained from Lemma \ref{l:aij-coef-Lp-very-soft}.
\begin{align*}
|\bar a_{ij}| &\leq C (1+|v|)^{2+\gamma+d-d/p} \|f\|_{L^p_\kappa} && \text{ since $p>d/(d+\gamma+2)$ and $\kappa \geq p(2+\gamma+d)-d$}, \\
\bar c(t,v) &\lesssim \|f\|_{L^p}^{p(d+\gamma)/d} m(t)^{1 - \frac{p(d+\gamma)} d}  && \text{ since } p < d / (d+\gamma)
\end{align*}

We prove the result applying Theorem \ref{t:local-max-pple-nonlinear} with $\beta = (2+\gamma+d-d/p)/2$ and $\alpha = 1-p(d+\gamma)/d$, and Corollary \ref{c:local-max-pple} as in the proof of Theorem \ref{t:landau}.

The restriction $\kappa \geq 2 + \frac {d^2} 2 \left(1 - \frac 1 p \right) - d \left( 1+\frac \gamma 2 \right) = \beta d - (d-1)(\gamma+2) - \gamma$ ensures that
\[ \det \bar a_{ij} \geq c (1+|v|)^{d \beta - \kappa}.\]

Note that because of the range of values of $p$, we have $\beta \in (0,1)$. Thus, we always have $\beta \in [-\kappa/d,1]$ because $\kappa\geq 0$.

\end{proof}

\begin{remark}
It is not clear if the value of $\kappa$ used in Theorem \ref{t:landau-conditional} is optimal. The requirement $p > d / (d+2+\gamma)$ seems to be optimal for the approach of this proof to work.
\end{remark}

\begin{remark}
If we have a bound for $\|f\|_{L^p_\kappa}$ for some $p > d/(d+\gamma)$, then this implies another bound for a smaller $p$ as well by interpolation between $\|f\|_{L^p_\kappa}$ and $\|f\|_{L^1_2}$. Thus, the restriction $p < d/(d+\gamma)$ is not limiting. In any case, a version of Theorem \ref{t:landau-conditional} can also be proved for $p > d/(d+\gamma)$. In this case, we would use the following upper bound for $\bar c$,
\[ \bar c \lesssim M_0^{-\frac{p\gamma}{(p-1)d}}  \|f\|_{L^p}^{1+\frac{p\gamma}{(p-1)d}}.\]
This bound is independent of $\|f\|_{L^\infty}$ and therefore we would apply Theorem \ref{t:local-max-pple-nonlinear} with $\alpha=0$, and a suitable value of $\kappa$.
\end{remark}

\bibliographystyle{plain}
\bibliography{landau8}
\end{document}